\documentclass[a4paper,11pt]{amsart}
\usepackage{amsmath,amsfonts,amscd,amssymb,amsthm}
\usepackage{a4wide}
\usepackage{enumerate}
\usepackage{longtable}
\usepackage[english]{babel}
\usepackage[utf8x]{inputenc}
\usepackage[active]{srcltx}
\usepackage[T1]{fontenc}
\usepackage{bbm}
\usepackage{leftidx}
\usepackage{MnSymbol}
\usepackage{stmaryrd}
\usepackage{latexsym}
\usepackage{mathtools}
\usepackage{caption}

\theoremstyle{plain}
\newtheorem*{theorem-non}{Theorem}
\newtheorem*{thm:eulerchar}{Theorem \ref{mresult}}
\newtheorem{theorem}{Theorem}[section]

\newtheorem{corollary}[theorem]{Corollary}
\newtheorem{lemma}[theorem]{Lemma}
\newtheorem{proposition}[theorem]{Proposition}

\theoremstyle{definition}
\newtheorem{definition}{Definition}[section]
\newtheorem{example}[theorem]{Example}
 
\newcommand{\be}[1]{\begin{equation}\label{#1}}
\newcommand{\ee}{\end{equation}}
\numberwithin{equation}{section}
   
\newcommand{\bbO}{\mathbb{O}}

\newcommand{\bbX}{\mathbb{X}}

\newcommand{\bbZ}{\mathbb{Z}}

\newcommand{\al}{\alpha}
\renewcommand{\be}{\beta}

\newcommand{\ep}{\varepsilon}
\newcommand{\sgn}{\operatorname{sgn}}

\begin{document}

\title{Link Floer homology categorifies the Conway function}
\author{Mounir Benheddi \and David Cimasoni}
\address{Universit\'e de Gen\`eve, Section de math\'ematiques, 2-4 rue du Li\`evre, 1211 Gen\`eve 4, Switzerland}
\email{Mounir.Benheddi@unige.ch}
\email{David.Cimasoni@unige.ch}

\begin{abstract}
Given an oriented link in the 3-sphere, the Euler characteristic of its link Floer homology is known to coincide with its multivariable Alexander polynomial,
an invariant only defined up to a sign and powers of the variables. In this paper, we get rid of this ambiguity by proving that this Euler characteristic is equal to the so-called Conway function, the representative of the multivariable Alexander polynomial introduced by Conway in 1970 and explicitly constructed by Hartley in 1983.
This is achieved by creating a model of the Conway function adapted to rectangular diagrams, which is then compared to the Euler characteristic of the combinatorial
version of link Floer homology.
\end{abstract}

\maketitle


\section{Introduction}

The Alexander polynomial is probably the most celebrated invariant of knots and links. Both the one-variable and the multivariable versions were introduced by Alexander in
his seminal 1928 paper~\cite{alexander1928topological}. The latter invariant associates to a
$\mu$-component oriented link~$L$ in the standard~$3$-sphere~$S^3$ a~$\mu$-variable Laurent polynomial~$\Delta_L (t_1,\ldots,t_\mu)\in\bbZ[t_1^{\pm 1},\ldots,t_{\mu}^{\pm 1}]$
defined up to units of this ring, that is, up to a sign and powers of the variables. The indetermination in powers of the variables is not problematic. Indeed, this polynomial
is palindromic, i.e.
\[
\Delta_L(t_1^{-1},\ldots,t_{\mu}^{-1})\overset{\centerdot}{=}\Delta_L(t_1,\ldots,t_{\mu})\,,
\]
where~$\overset{\centerdot}{=}$ denotes equality up to multiplication by units; hence, a natural representative in the ring~$\bbZ[t_1^{\pm 1/2},\ldots,t_{\mu}^{\pm 1/2}]$
can be chosen up to a sign. If~$L$ is a knot, then the Alexander polynomial further satisfies~$\Delta_L(1)=\pm 1$, so it can be normalized by setting this value to be~$+1$.
In the link case however, this sign issue is far from trivial.

In 1970, Conway~\cite{conway1970enumeration} suggested a natural representative of the multivariable Alexander polynomial, later called the \emph{Conway function}.
Given a~$\mu$-component oriented link~$L$, its Conway function is a rational function~$\nabla_L(t_1,\ldots,t_{\mu})$ such that
\[
\nabla_L(t_1,\ldots,t_{\mu}) \overset{\centerdot}{=}
\begin{dcases*}
(t_1-t_1^{-1})\Delta_L (t_1^2) & if~$\mu=1$ \\
\Delta_L (t_1^2,\ldots,t_\mu^2) & if~$\mu >1$\,. \\
\end{dcases*}
\]
Conway also noticed several properties of this invariant via local transformations of the link, including the celebrated {\em skein relation\/}.
However, a precise definition of this invariant together with a proof of its various properties only appeared more than a decade later. In 1983, Hartley~\cite{hartley1983conway}
gave a model for the Conway function as a well chosen normalization of the determinant of a Fox matrix obtained from a Wirtinger presentation of the link group via Fox free calculus.
Other constructions were later given by Turaev~\cite{turaev1986reidemeister} using sign-refined Reidemeister torsion, and by the second
author~\cite{cimasoni2004geometric} using generalized Seifert surfaces.

In 2004, Ozsváth and Szabó~\cite{ozsvath2004holomorphic} introduced Heegaard-Floer homology as an invariant for closed oriented~$3$-manifolds. They later extended this theory
to give an invariant for null-homologous oriented knots in such manifolds~\cite{ozsvath2004holomorphic0}, a construction further generalized to oriented
links~\cite{ozsvath2008holomorphic}. In its most basic form, this \emph{link Floer homology} is a finite-dimensional bi-graded vector space over~$\bbZ_2$
\[
\widehat{\mathit{HL}}(L)= \bigoplus_{d,s} \widehat{\mathit{HL}}_d(L,s)\,,
\]
the direct sum ranging over all~$d\in\bbZ$ and~$s=(s_1,\ldots,s_{\mu})\in\left(\frac{1}{2}\bbZ \right)^{\mu}$.
All these invariants were originally defined using Heegaard diagrams and count of holomorphic discs in the symmetric product of a Riemann surface.
In~\cite{manolescu2006combinatorial, manolescu2007combinatorial}, a purely combinatorial description of these invariants was provided in the
case of links in~$S^3$, relying on the count of rectangles in so-called \emph{grid diagrams\/} for these links.

One of the most fundamental properties of link Floer homology is that it {\em categorifies\/} the multivariable Alexander polynomial. More precisely, its Euler characteristic
\[
\chi(\widehat{\mathit{HL}}(L);t_1,\dots,t_\mu) := \sum\limits_{d,s} (-1)^d \mbox{dim}(\widehat{\mathit{HL}}_d(L,s))\,t_1^{s_1}\cdots t_\mu^{s_\mu}
\]
satisfies the equality
\[
\chi(\widehat{\mathit{HL}}(L);t_1,\dots,t_\mu)\overset{\centerdot}{=}
\begin{dcases*}
\prod\limits_{k=1}^{\mu}(t_k^{1/2}-t_k^{-1/2})\, \Delta_L(t_1,\dots,t_\mu) & if~$\mu>1$ \\
\Delta_L(t_1) & if~$\mu=1$\,.
\end{dcases*}
\]
Therefore, the Euler characteristic of link Floer homology provides a natural representative for the multivariable Alexander
polynomial, and it is natural to ask whether it coincides with the Conway function.

\medskip

In the present article, we give a positive answer to this question.

\begin{theorem}\label{mresult}
Given a~$\mu$-component oriented link~$L$ in~$S^3$, the equation
\[
\chi(\widehat{\mathit{HL}}(L);t_1,\dots,t_\mu)=\prod\limits_{k=1}^{\mu}(t_k^{1/2}-t_k^{-1/2}) \, \nabla_L(t_1^{1/2},\ldots,t_{\mu}^{1/2})
\]
holds in~$\bbZ[t_1^{\pm 1/2},\ldots,t_{\mu}^{\pm 1/2}]$.
\end{theorem}

We also sketch a proof of the fact that Turaev's model coincides with the second author's geometric model, which was shown to agree with Hartley's model
in~\cite{cimasoni2004geometric}. Therefore, we are in the presence of four different constructions of the same invariant. The main interest in identifying them is that some
properties of the Conway function are totally transparent in one model, and quite surprising in others. For example, the geometric model is very well suited
for discovering skein-type relations for the Conway function, whose translation in terms of link Floer homology is not always obvious. Also, the surgery formula obtained
by Boyer-Lines in~\cite{boyer1992conway} using Hartley's model does not appear to be known in link Floer homology.

It should be mentioned that even though Hartley's model and the combinatorial version of link Floer homology are only defined for links in
the standard~$3$-sphere, the models of Turaev and of the second author are defined for links in an arbitrary integral homology~$3$-sphere, and so is link Floer homology.
The question of whether these models coincide for links in integral homology spheres is addressed in a slightly informal way in the last section of the present article.

\medskip

This paper is organized as follows. In Section~\ref{hartley}, we briefly recall Hartley's construction of the Conway function~$\nabla_L$ and amend it to define a model
$\Gamma_L$ adapted to rectangular diagrams. We then show that~$\Gamma_L$ satisfies Jiang's five characterizing relations~\cite{jiang2014conway} and therefore coincides with
$\nabla_L$. In Section~\ref{grid}, we define grid diagrams, recall their use in the combinatorial definition of link Floer homology, and identify~$\chi(\widehat{\mathit{HL}}(L))$
with~$\Gamma_L$. Finally, Section~\ref{models} discusses the identification with other models as well as the extension to homology spheres.

\subsubsection*{Acknowledgments} 
The authors thankfully acknowledge support by the Swiss National Science Foundation. The second author also wishes to express his thanks to Steven Boyer and Vladimir Turaev
for useful exchanges.


\section{A model for the Conway function using rectangular diagrams}
\label{hartley}

This section focuses on adapting Hartley's construction to a special type of diagrams, called rectangular diagrams. We first recall Hartley's definition~\cite{hartley1983conway},
and introduce rectangular diagrams as well as Neuwirth's associated presentation of the link group. Then we turn to the definition of our model~$\Gamma_L$, show its invariance,
and identify it with Hartley's model.

\subsection{Hartley's construction of~$\nabla_L$}
\label{sub:hartley}

Let~$L=L_1 \cup \cdots \cup L_{\mu}$ be a~$\mu$-component oriented link in~$S^3$, and let~$W=\left<x_1,\ldots, x_n \, | \, r_1,\ldots, r_{n}\right>$ be a Wirtinger presentation of
$\pi_1(S^3 \backslash L)$. Recall that the generators (resp. the relations) are in one-to-one correspondence with the arcs (resp. the crossings) of the diagram.
Let~$\theta\colon\bbZ[\pi_1(S^3 \backslash L)]\rightarrow \bbZ[t_1^{\pm 1},\ldots, t_{n}^{\pm 1}]$ be the extension of the abelianization morphism,
i.e. the unique morphism of rings mapping~$x_j$ to~$t_k$ if the arc corresponding to~$x_j$ belongs to~$L_k$.
For all~$1\le i\le n$, consider a path from a fixed point outside the diagram to a point lying at the right of both the over-crossing and under-crossing arcs of the
$i^\mathrm{th}$ crossing; this path intersects some arcs of the diagram and thus gives a word~$u_i$ in the~$x_j$'s. Define
\[
D(t_1,\ldots,t_\mu)= \frac{(-1)^{i+j}\mbox{det}(M^{(ij)})}{\theta(u_i)(\theta(x_j)-1)}\,,
\]
where~$M=\left(\theta\left(\frac{\partial r_i}{\partial x_j}\right) \right)_{i,j}$ is the image under~$\theta$ of the Jacobian matrix of~$W$ with respect to free differential calculus (see~\cite{fox1954free}), and~$M^{(ij)}$ is~$M$ with the~$i^\mathrm{th}$ row and~$j^\mathrm{th}$ column removed. This function~$D$ is a rational function which does not depend on the choice of~$i$ and~$j$, but still needs to be normalized to give a well-defined invariant. This is done as follows.

For each component~$L_k$ of the link, consider the projection of that component and resolve every crossing as illustrated below.
\begin{center}
\includegraphics{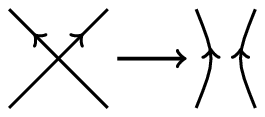}
\end{center}
Once this process is completed, one obtains a collection of oriented circles. The \emph{curvature}~$\kappa_k$ of~$L_k$ 
is defined as the number of anticlockwise circuits minus the number of clockwise circuits. Finally let~$\eta_k$ be the number of crossings where the over-crossing arc is~$L_k$.

The \emph{Conway function} of~$L$ is the rational function given by
\[
\nabla_L(t_1,\ldots,t_\mu):= D(t_1^2,\ldots,t_\mu^2)\, t_1^{\kappa_1-\eta_1}\cdots t_\mu^{\kappa_\mu-\eta_\mu}\,.
\]
As checked by Hartley~\cite{hartley1983conway}, it is a well-defined link invariant. Furthermore, it was recently established by Jiang~\cite{jiang2014conway}
that it is characterized by a relatively simple set of five relations (see paragraph~\ref{proof} below).

\subsection{Rectangular diagrams}
\label{rect}

\begin{definition}
A \emph{rectangular diagram} is a link diagram~$D$ with the conditions that:
\begin{enumerate}[(i)]
\item{The diagram is composed of horizontal and vertical line segments only.}
\item{At each crossing, the vertical segment is the over-crossing and the horizontal one the under-crossing.}
\item{No two segments are collinear.}
\end{enumerate}
\end{definition}

Clearly, any link can be represented by a rectangular diagram. For example, the transformation illustrated below can be used to ensure that condition (ii) is satisfied.
\begin{center}
\includegraphics{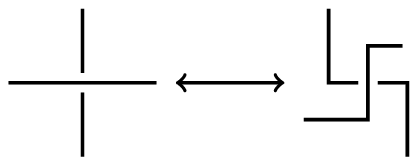}
\end{center}
Note that because of this condition, there is no need to differentiate the over-crossings and under-crossings
in such a diagram. Also condition (iii) allows for the diagram to be included in a grid whose rows and columns contain respectively a unique horizontal and vertical segment.
We shall call the turning points of segments \emph{corners}; there are four types of corners, that all appear in the example for the trivial knot given in
Figure~\ref{gammacalc}. 

\smallskip

Cromwell~\cite{cromwell1995embedding} proved that any two rectangular diagrams describing the same link can be connected through a finite sequence of the following
elementary moves and their inverses.

\begin{enumerate}[$(G1)$]
\item{Cyclic permutation of vertical (resp. horizontal) segments: the leftmost (resp. topmost) segment is moved to the rightmost (resp. bottommost) position.
}
\item{Commutation of horizontal or vertical segments: two adjacent columns or rows can be permuted if, when one is projected onto the other, their images are either disjoint or one is completely contained in the other.
}
\item{Stabilization: at a selected corner of the diagram, add a row and a column next to the corner. There is a choice of adding the new column to the right or left of the selected column and a choice of over or under the selected row for the new one, which yields four different stabilizations for each type of corner.
Using the previous moves, it is enough however to consider that operation for one type of corner only.
}
\end{enumerate}

There exists a presentation of the link group suited for rectangular diagrams, introduced by Neuwirth in~\cite{neuwirth1984proj}. The generators are in one-to-one correspondence with the vertical segments of the diagram (numbered from left to right). Suppose there are~$n$ vertical segments, then there are also~$n$ horizontal segments. This presentation has~$n-1$ relations, one for each \emph{line} in the diagram, i.e. for each position between two consecutive horizontal segments. Consider a path crossing the diagram at constant height between the~$i^\mathrm{th}$ and~$(i+1)^\mathrm{th}$ horizontal segments (counted from the top). In terms of generators it corresponds to the product of the vertical segments crossed.
This product is the~$i^\mathrm{th}$ relation~$r_i$.  

This presentation is well behaved under the action of Fox's free derivatives: indeed, any relation is of the form~$r=x_{j_1} \cdots x_{j_m}$ with different~$j_\ell$'s, so
\[
\frac{\partial r}{\partial x_j} =
\begin{dcases*}
x_{j_1}\cdots x_{j_{\ell-1}} & if ~$j=j_\ell$,~$1\le \ell \le m$,\\
0 & otherwise.
\end{dcases*}
\]
Given a rectangular diagram~$D$, define the associated Fox matrix as the~$(n-1)\times n$ matrix~$F_D$ with~$(i,j)$-coefficient~$\theta\left(\frac{\partial r_i}{\partial x_j} \right)$,
with~$\theta$ the abelianization morphism as above.
Note that the non-zero entries of~$F_D$ correspond to the intersections of the vertical segments of the diagram with the lines of the diagram.
Furthermore, such a non-zero coefficient at position~$(i,j)$ is given by~$t_1^{\varphi_1}\cdots t_\mu^{\varphi_\mu}$, where~$\varphi_k$ is the algebraic intersection number of 
$L_k$ with the~$i^{\mathrm{th}}$ line left of the~$j^{\mathrm{th}}$ vertical segment. (Here, an intersection point is counted positively if~$L_k$ is oriented upwards.) Note the equality~$\sum_j\frac{\partial r_i}{\partial x_j}(x_j-1)=r_i-1$, which implies
\begin{equation}\label{FFDC}
\sum_{j=1}^n(\theta(x_j)-1)\theta\left(\frac{\partial r_i}{\partial x_j}\right)=0
\end{equation}
for all~$1\le i\le n$.

\begin{example}\label{example1}
Examples of rectangular diagrams for the unknot~$K$, the Hopf link~$H$ and the trefoil~$T$ are illustrated in Figure~\ref{gammacalc}.
The corresponding Fox matrices are given by
\[
F_K=\left(\begin{matrix}1 & t^{-1}\end{matrix}\right),
\quad
F_T =
\left(
\begin{matrix}
1       & 0     & 0     & t^{-1} & 0 \\
1       & 0     & t^{-1} & 1   & t \\
1       & t^{-1} & 1    & 0     & t \\
0       & 1      & 0     & 0    & t \\
\end{matrix}
\right),
\quad\text{and}\quad
F_H =
\left(
\begin{matrix}
1       & 0     & t_1^{-1}      & 0 \\
1       & t_1^{-1}      & t_1^{-1}t_2^{-1} & t_2^{-1}  \\
0       & 1   & 0       & t_2^{-1} \\
\end{matrix}
\right)\,.
\]
\end{example}

\begin{figure}[Htb]
\begin{center}
\includegraphics{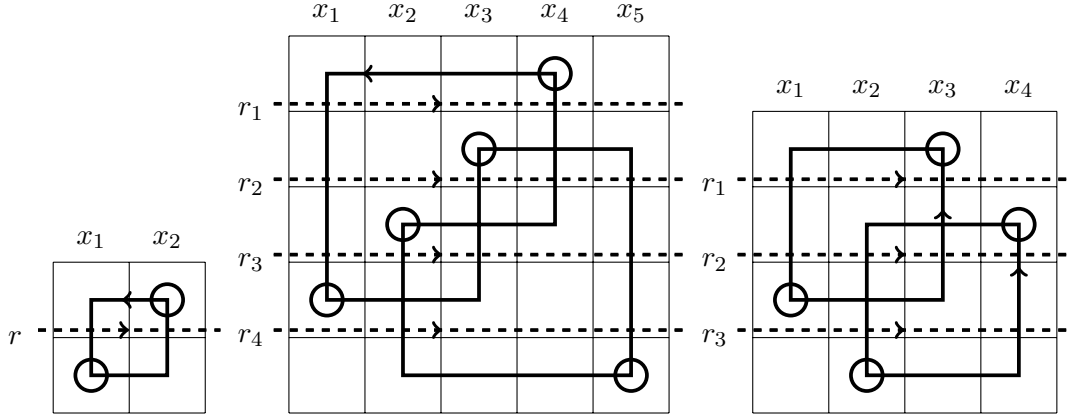}
\caption{Three examples of rectangular diagrams for the trivial knot, the trefoil, and the Hopf link.
The corresponding bases are given by the permutations~$x_0=(12)$,~$x_0=(14)(23)$ and~$x_0=(13)(24)$.}
\label{gammacalc}
\end{center}
\end{figure}

\subsection{Definition and invariance of the model~$\Gamma_L$}\label{sub:gamma}

From this section onwards, we will use the notations~$t=(t_1,\ldots, t_\mu)$, and~$t^{s}=\prod_{k=1}^{\mu}t_k^{s_k}$ for
$s=(s_1,\ldots,s_{\mu})\in\left(\frac{1}{2}\bbZ\right)^{\mu}$. Additionally, the vector~$(b,\ldots,b)\in\left(\frac{1}{2}\bbZ\right)^{\mu}$ will be denoted by~$\bf{b}$.

Given an oriented link~$L$ given by a rectangular diagram~$D$, let us write 
\[
M(t)= \frac{(-1)^{j}}{\theta(x_j)-1}\mbox{det}(F_D^{j})\,,
\]
where~$F_D^{j}$ is~$F_D$ with the~$j^\mathrm{th}$ column removed. The fact that~$M(t)$ is independent of the choice of~$j$ follows from Equation~(\ref{FFDC}).

Each corner is either the beginning or the end of a horizontal segment;
the former ones will be emphasized on the diagrams by small circles. These~$n$ circles define a permutation~$x_0\in S_n$, called the \emph{base of the diagram}, in the following way.
Number the rows top to bottom and the columns left to right from~$1$ to~$n$. The rows represent the domain of the permutation, the columns the co-domain, and the pairs~$(i,x_0(i))$
are given by the position of the circles. Examples are given in Figure~\ref{gammacalc}.

Note that if a corner is the intersection of the~$i^\mathrm{th}$ horizontal and~$j^\mathrm{th}$ vertical segments, then only one of~$(F_D)_{i-1,j}$ or~$(F_D)_{i,j}$ is
non-zero. We will call this coefficient the \emph{weight} of that corner. The \emph{total weight}~$\omega$ of the diagram is defined as the product of the weights of all
the corners of the diagram. 

\begin{definition}\label{def:gamma}
Given an oriented rectangular diagram~$D$, define
\[
\Gamma_D(t) := \frac{\sgn(x_0)(-1)^{u}}{\omega \, t^{\kappa}}M(t^{\bf{2}})\,,
\]
where~$\sgn(x_0)\in\{\pm 1\}$ is the signature of the base~$x_0\in S_n$,~$u$ the number of upwards segments,~$\omega$ the total weight,
$\kappa=(\kappa_1,\dots,\kappa_\mu)$ the curvature, and~$M(t)$ the above normalization of the determinant of the Fox matrix associated with~$D$.
\end{definition}

\begin{example}
\label{ex2}
Building on Example~\ref{example1}, let us compute the value of~$\Gamma_D$ for the diagrams~$K$,~$H$ and~$T$ of Figure~\ref{gammacalc}.
For the trivial knot~$K$,~$\sgn(x_0) = -1$. Using~$F_K$, one computes~$\omega=t^{-2}$. From the diagram, we see that~$u=\kappa=1$.
Removing the first column of~$F_K$ yields~$M(t) = -\frac{t^{-1}}{t^{-1}-1}$ and hence:
\[
\Gamma_{K}(t)=\frac{(-1)(-1)}{t^{-2}t^{1}}\,\frac{-t^{-2}}{t^{-2}-1} = \frac{1}{t-t^{-1}}\,.
\
\]
For the trefoil~$T$,~$\sgn(x_0) = 1$,~$\omega = t^{-1}$,~$u=3$ and~$\kappa=0$. The determinant of~$F_T^{1}$ equals~$-t^{-2} + t^{-1} - 1$, therefore:
\[
\Gamma_{T}(t)= (t-t^{-1})^{-1}(t^2 - 1 + t^{-2})\,.
\]
For the Hopf link~$H$,~$\sgn(x_0)=1$,~$\omega=t^{\mathbf{-3}}$,~$u=3$ and~$\kappa = t^{\mathbf{1}}$. The determinant of~$F_H^{1}$ is given by~$-t_1^{-1}t_2^{-1}(t_1^{-1}-1)$, so
\[
\Gamma_{H}(t) = 1\,.
\]
\end{example}

\begin{proposition}\label{gammainv}
The function~$\Gamma_{D}$ is an invariant of oriented links.
\end{proposition}

\begin{proof}
By Cromwell's theorem, we only need to check that~$\Gamma_D$ is left unchanged by the moves~$G1$,~$G2$ and~$G3$ described above.
Throughout the proof, the unprimed quantities refer to the diagrams on the left (``before the move'') and the primed ones to those on the right (``after the move'').
The first pictured line of a diagram is always the~$\ell^\mathrm{th}$, and the first column the~$q^\mathrm{th}$. Set~$\ep_j=1$ if~$x_j$ is oriented upwards and~$\ep_j=-1$ otherwise.

\smallskip

\textbf{(G1)} Consider first the case of vertical cyclic permutation, illustrated below.
\medskip
\begin{center}
\includegraphics{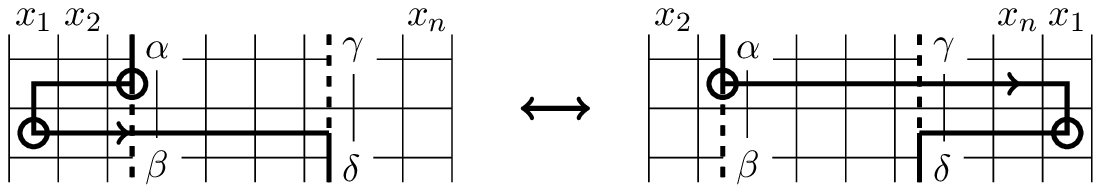}
\end{center}
Without loss of generality, assume that the moved segment belongs to~$L_1$ and set~$\ep=\ep_1$. If there are~$m$ rows of the grid affected by the move, then there are~$m-1$ relations that change. (In the illustration above,~$m=2$.) This leads to the equality
$\theta\left(\frac{\partial r_i'}{\partial x_j}\right)= t_1^{-\ep}\theta\left(\frac{\partial r_i}{\partial x_j}\right)$ for
these~$m-1$ indices~$i$ and for all~$j$'s. Computing~$M(t)$ by removing the first column of~$F$ and~$M'(t)$ by removing the last column of~$F'$, we get
\[
M'(t)=(-1)^{n-1}t_1^{-(m-1)\ep}M(t)\,.
\]
Note that the two bases~$x_0$ and~$x_0'$ are related by a cyclic permutation of order~$n$, so~$\sgn(x_0')=(-1)^{n-1}\sgn(x_0)$. Also, the number of upwards segments
is left unchanged by this move.

Let us now analyze how the corner weights and curvature are modified.
There are~$2m$ corners whose weight might change. The~$2m-2$ ``interior corners'' will have their weight multiplied by~$t_1^{-\ep}$. As for the two remaining ones,
the result will depend on the position of the vertical segment at each of these two corners. There are two possible configurations for the left one (denoted by~$\alpha$ and~$\beta$), and 
two for the right one ($\gamma$ and~$\delta$) which are indicated in the picture above:~$\al$ and~$\delta$ as solid lines,~$\beta$ and~$\gamma$ as dashed lines.
In the configuration~$\al$ (resp.~$\delta$) the weight of the left (resp. right) corner does not change, while in the configuration~$\beta$ (resp.~$\gamma$), the weight of left
(resp. right) corner is multiplied by a factor~$t_1^{-\ep}$. Note finally that the curvature is multiplied by~$t_1^\ep$ (resp.~$1$,~$t_1^{2\ep}$,~$t_1^{\ep}$) in the configuration
$\alpha\gamma$ (resp.~$\alpha\delta$,~$\beta\gamma$,~$\beta\delta$). It appears that in every case, the factor~$\omega t^{\kappa}$ changes by a factor~$t_1^{-(2m-2)\ep}$.
Therefore, all the changes cancel out and~$\Gamma_{D'}=\Gamma_{D}$.

Now consider the horizontal cyclic permutation, with~$m$ columns of the grid affected, illustrated below in the case~$m=2$.
\begin{center}
\includegraphics{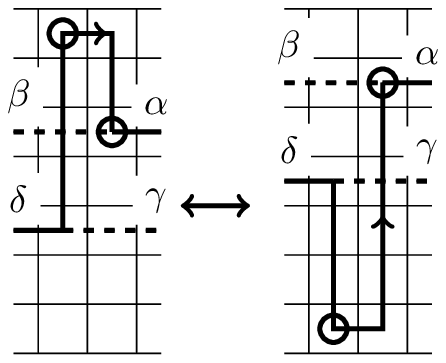}
\end{center}
In that case, the Fox matrix changes only between the columns~$q$ and~$q+m-1$. Let~$A$ be the matrix obtained from~$F$ by removing the first line and those two columns.
A quick computation gives~$\det(F^{q})= (-1)^{q}t_1^{\ep}\det(A)$ and~$\det(F'^{q+1}) = (-1)^{q+n}t_1^{-(m-2)\ep}\det(A)$, so
\[
M'(t) = (-1)^{n-1}t_1^{-(m-1)\ep}M(t)\,.
\]
Here again, we have~$\sgn(x_0')=(-1)^{n-1}\sgn(x_0)$ and~$u'=u$.
For the total weight and curvature, one easily checks that in each of the four possible configurations,~$\omega t^{\kappa}$ is multiplied by a factor~$t_1^{-(2m-2)\ep}$.
The equality~$\Gamma_{D'}=\Gamma_D$ follows.

\smallskip

\textbf{(G2)} Let us first consider the case of a vertical commutation with disjoint segments. 
The only quantities that change are the signature of the basis, multiplied by~$-1$, and the determinant, also multiplied by~$-1$, so~$\Gamma_{D'} = \Gamma_D$.

Let us now assume that the segments are not disjoint, as in the following illustration.
\begin{center}
\includegraphics{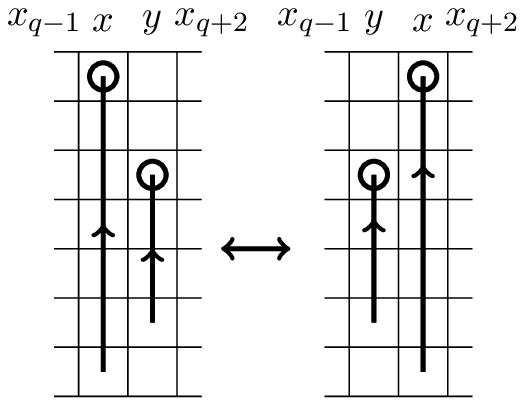}
\end{center}
Clearly, the signature of the base is multiplied by~$-1$, while~$u'=u$ and~$\kappa'=\kappa$.
Let us denote by~$L_k$ (resp.~$L_l$) the component of~$L$ to which the vertical segment~$x$ (resp.~$y$) belongs. The parts of the Fox matrices affected by the move are of the form
\[
F=
\begin{pmatrix}
\ast & 0 \\
\al_i & \al_it_k^{\ep_x} \\
\vdots & \vdots \\
\al_{i+p} & \al_{i+p}t_x^{\ep_x} \\
\ast & 0 \\
\end{pmatrix}
\quad\text{and}\quad
F'=
\begin{pmatrix}
0 & \ast  \\
\al_i & \al_it_l^{\ep_y} \\
\vdots & \vdots \\
\al_{i+p} & \al_{i+p}t_l^{\ep_y} \\
0 & \ast  \\
\end{pmatrix}\,.
\]
Subtracting from the~$q^\mathrm{th}$ column of~$F$ its~$(q+1)^\mathrm{th}$ column multiplied by~$t_k^{-\ep_x}$, and from the~$(q+1)^\mathrm{th}$ column of~$F'$
its~$q^\mathrm{th}$ column multiplied by~$t_l^{\ep_y}$ leads to the relation~$\det(F'^{q}) = - t_k^{-\ep_{x}}\det(F^{q})$.
As for the corner weights, only the two corners of the small segment change their contribution, yielding~$\omega'=\omega t_k^{-2\ep_x}$.
The equality~$\Gamma_{D'}=\Gamma_D$ follows.

Consider now the horizontal commutation with disjoint segments, pictured below.
\begin{center}
\includegraphics{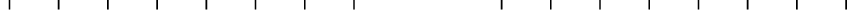}
\end{center}
In that case, the signature of the base is multiplied by~$-1$, while~$u'=u$ and~$\kappa'=\kappa$. Furthermore, one easily sees that
\[
-\theta\left(\frac{\partial r_{\ell}}{\partial x_j}\right) = \theta\left(\frac{\partial r_{\ell}'}{\partial x_j}\right)
-\theta\left(\frac{\partial r_{\ell-1}'}{\partial x_j}\right)-\theta\left(\frac{\partial r_{\ell+1}'}{\partial x_j}\right)
\]
for all~$j$'s. This implies that~$\det(F'^1)=-\det(F^1)$, leading to the expected invariance.

Let us now turn to the case where the segments are not disjoint, illustrated below.
\begin{center}
\includegraphics{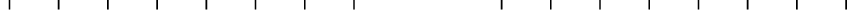}
\end{center}
In this case, only the base and the matrix change.
Let~$a,b,c,d$ be the vertical segments at the extremities of the horizontal ones as pictured above. We obtain 
\[
-t_k^{\ep_a}\theta\left(\frac{\partial r_{\ell}}{\partial x_j}\right) = \theta\left(\frac{\partial r_{\ell}'}{\partial x_j}\right)
-t_k^{\ep_a}\theta\left(\frac{\partial r_{\ell-1}'}{\partial x_j}\right)-\theta\left(\frac{\partial r_{\ell+1}'}{\partial x_j}\right)
\]
for all~$j$'s, where~$a$ belongs to~$L_k$. This leads to the equality~$\det(F'^1)=-t_k^{-\ep_a}\det(F^1)$.
The only corners affected by the change in the matrix are those of the small segment: they are both multiplied by~$t_k^{\ep_a}$, so all the changes cancel out once again.

\smallskip

\textbf{(G3)}  Consider the first stabilization pictured below.
\begin{center}
\includegraphics{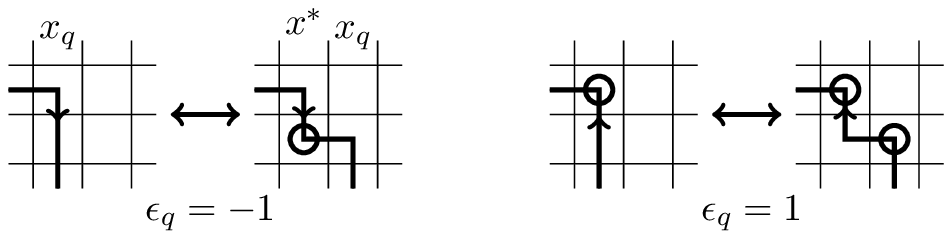}
\end{center}
In the case~$\ep_q=-1$, the move adds a circle at the~$(\ell+1,q)$ position of the grid, so~$\sgn(x_0')=(-1)^{(\ell+q+1)}\sgn(x_0)$, while the number of upwards segments
does not change. In the case~$\ep_q=1$, the move adds a circle at the~$(\ell+1,q+1)$ position of the grid, so~$\sgn(x_0')=(-1)^{(l+q)} \sgn(x_0)$, while~$u'=u+1$.
Thus in both cases, the overall sign changes by a factor~$(-1)^{(\ell+q+1)}$.
There are two new corners, each with a weight~$\alpha:=\theta\left(\frac{\partial r^\ast}{\partial x^{\ast}}\right)$, where~$x^\ast$ and~$r^\ast$ denote the new generator and relation
created by the stabilization. This implies~$\omega'=\omega\al^2$.
The matrix~$F'$ is obtained from~$F$ by adding a line and a column; developing the determinant with respect to that column yields
$\det(F'^{1}) = (-1)^{(\ell+q+1)}\al\det(F^{1})$, so all the changes cancel out and~$\Gamma_{D'}= \Gamma_D$.

The three remaining types of stabilizations are illustrated below.
\begin{center}
\includegraphics{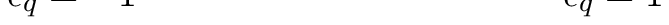}
\includegraphics{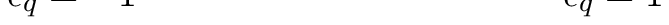}
\includegraphics{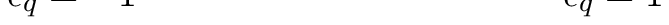}
\end{center}
They can be treated in the same way.
\end{proof}

\subsection{Identification of~$\Gamma_L$ and~$\nabla_L$}
\label{proof}

By Proposition~\ref{gammainv},~$\Gamma_D$ is an invariant and can therefore be denoted by~$\Gamma_L$. The aim of this paragraph is to prove the following result.

\begin{proposition}
\label{gammavsnabla}
For any oriented link~$L$, the invariants~$\Gamma_L$ and~$\nabla_L$ coincide. In particular,~$\Gamma_L$ satisfies the equality
\begin{equation}\label{gammasym}
\Gamma_L(t^{-1})=(-1)^\mu\Gamma_L(t)\,.
\end{equation}
\end{proposition}

To verify this statement, we first used Murakami's characterization theorem~\cite{murakami1992local}, which states that~$\nabla_L$ is determined uniquely by a system
of six local relations. However Jiang~\cite{jiang2014conway} recently showed that Conway's function is characterized by the following (simpler) set of five relations, which we
will use in our proof.
\[
\tag{$R1$}\nabla_H = 1\,,
\]
where~$H$ is the positive Hopf link.
\[
\tag{$R2$}\nabla_{L \sqcup K} = 0\,,
\]
where~$L \sqcup K$ denotes the disjoint union of~$L$ and a trivial knot~$K$.
\[
\tag{$R3$}\nabla_{L'} = (t_i -t_i^{-1})\,\nabla_{L_0}\,,
\]
where~$L'$ is obtained from~$L_0$ by the local operation given below.
\begin{center}
\includegraphics{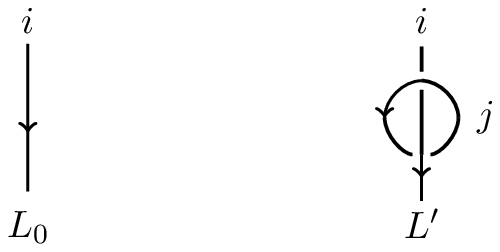}
\end{center}
\medskip
\[
\tag{$R4$}\nabla_{L_{++}} + \nabla_{L_{--}} = (t_it_j -t_i^{-1}t_j^{-1})\,\nabla_{L_0}\,,
\]
where~$L_{++}$,~$L_{--}$ and~$L_0$ differ by the following local operation.
\medskip
\begin{center}
\includegraphics{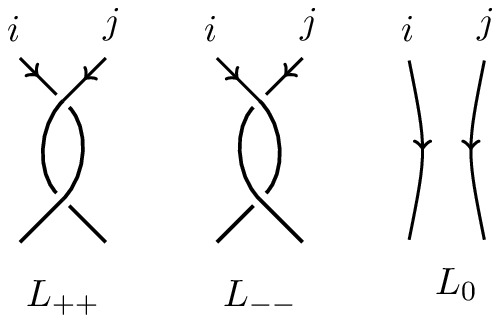}
\end{center}
\medskip
\[
\tag{$R5$}(t_it_k^{-1}-t_i^{-1}t_k^1)(\nabla_{L(0)} + \nabla_{L(1)})+(t_jt_k - t_j^{-1}t_k^{-1})(\nabla_{L(2)} + \nabla_{L(3)})
			+(t_i^{-1}t_j^{-1} - t_it_j)(\nabla_{L(4)} + \nabla_{L(5)})=0\,,
\] 
where~$L(0)$,~$L(1)$,~$L(2)$,~$L(3)$,~$L(4)$, and~$L(5)$ differ by the following local operation.
\medskip
\begin{center}
\includegraphics{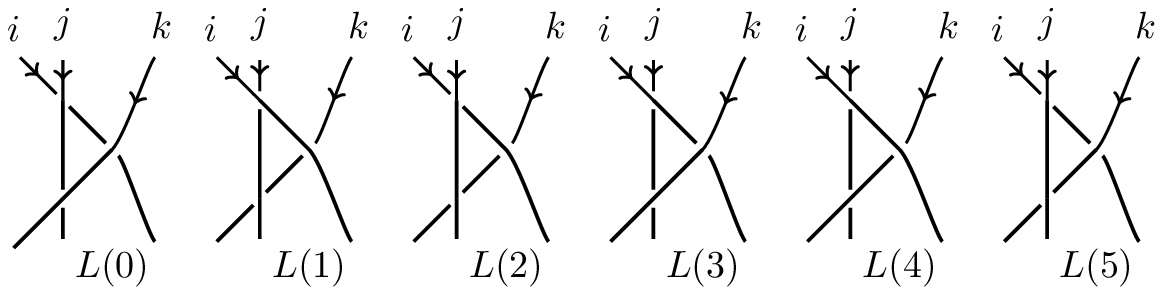}
\end{center}

\begin{proof}[Proof of Proposition~\ref{gammavsnabla}]
By Jiang's theorem, we only need to check that~$\Gamma_L$ satisfied the five relations described above. Let us first fix some notations.
In the grid diagrams illustrated below, the first horizontal pictured segment will always be the~$\ell^\mathrm{th}$, and the first vertical one the~$q^\mathrm{th}$.
Only the ``local'' Fox matrices will be written, i.e. the parts where the Fox matrices differ. The simplest diagram will be called~$L_0$ and used as a reference for the weight, curvature, base, matrices and number of upwards segments, with the associated quantities indexed by a zero. The quantities for the other diagrams will be indexed according to their names.

\smallskip

\textbf{(R1)} This property was verified in Example~\ref{ex2}.

\smallskip

\textbf{(R2)}
A possible rectangular diagram for~$L \sqcup K$ is given below.
\begin{center}
\includegraphics{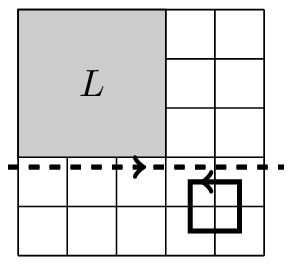}
\end{center}
The line in the Fox matrix corresponding to the dashed line illustrated above is identically zero. Therefore any~$(n-1)$-minor of the matrix vanishes and so does~$\Gamma_{L \sqcup K}$.

\smallskip

\textbf{(R3)} This relation compares the value of~$\Gamma_L$ for the two following diagrams.
\begin{center}
\includegraphics{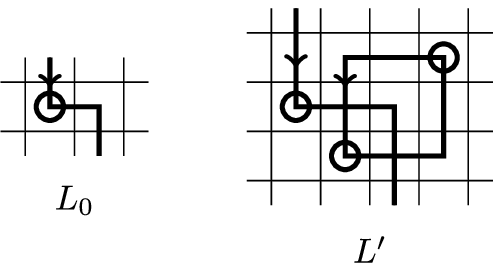}
\end{center}
Fix~$i=1$ and~$j=2$. The (local) Fox matrices are of the form
\[
F_0^q=
\begin{pmatrix}
  0  \\
 \al \\
\end{pmatrix},
\hspace{1cm}
F'^q=
\begin{pmatrix}
	0	&	0	&	0	\\
 \al t_1^{-1} 	&	0	& \al t_1^{-1}t_2^{-1} \\
 \al   &   \al t_2^{-1}  &   \al t_2^{-1}t_1^{-1} \\
 0   &  \al   &   0  \\	
\end{pmatrix}\,.
\]
This yields the equality~$\mbox{det}(F'^q)=\al^2t_1^{-1}t_2^{-1}(1-t_1^{-1})\,\mbox{det}(F^q)$.
Clearly,~$\kappa'=\kappa + (0,1)$ and~$u'=u_0+1$, and one easily checks that~$\sgn(x')=-\sgn(x_0)$.
The total weights for~$L'$ and~$L$ are related by~$\omega'=\omega_0\al^4t_1^{-3}t_2^{-3}$. Combining all these equalities, we get
\[
\Gamma' = \frac{(-1)(-1)}{\al^4t_1^{-3}t_2^{-3}t_2}\al^4t_1^{-2}t_2^{-2}(1-t_1^{-2})\Gamma_0 = (t_1 - t_1^{-1})\Gamma_0\,.
\]

The remaining two relations require more care. Given a line of a matrix, consider the determinant preserving operation
given by subtracting to this line an adjacent line. The resulting line is zero everywhere except for the entries separated by a horizontal
segment of the diagram. Therefore one can expand the determinant with respect to that line and all the information is contained in the local matrix.

\smallskip

\textbf{(R4)} Fix~$i=1$ and~$j=2$. Possible rectangular diagrams for this local operation are pictured below.
\begin{center}
\includegraphics{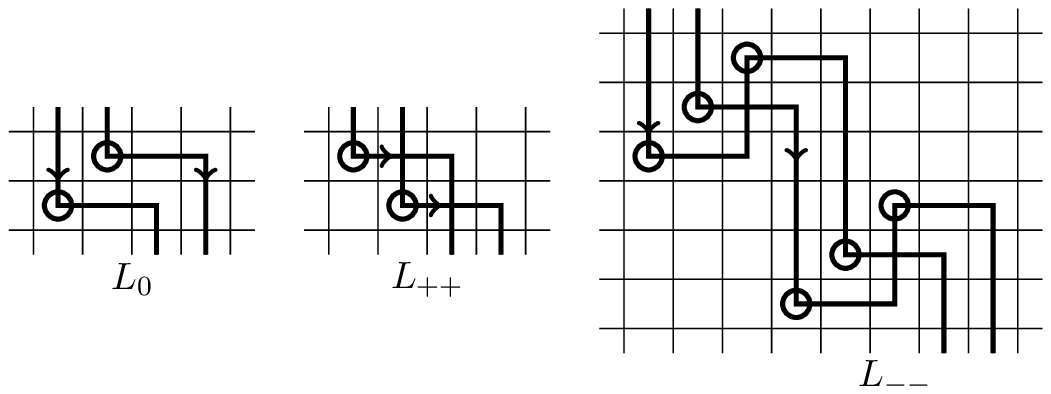}
\end{center}
The corresponding Fox matrices are of the following form:
\[
F_0^q=
\begin{pmatrix}
c_1 & 0 & 0 \\
\al t_1^{-1} & 0 & 0 \\
0 & 0 & \al t_1^{-1} \\
0 & \al & \al t_1^{-1} \\
0 & c_2 & c_3
\end{pmatrix},\quad
F_{++}^q=
\begin{pmatrix}
c_1 & 0 & 0 \\
\al t_1^{-1} & 0 & 0 \\
\al & \al t_2^{-1} & 0 \\
0 & \al & \al t_1^{-1}\\
0 & c_2 & c_3
\end{pmatrix},\quad \text{and}
\]
\[
F_{--}^q=
\begin{pmatrix}
c_1          &  0 & 0 & 0 & 0 & 0 & 0 \\
\al t_1^{-1} &	0 & 0 & 0 & 0 &	0 & 0 \\
\al t_1^{-1} & \al t_1^{-1}t_2^{-1} & 0 & \al t_2^{-1} & 0 & 0 & 0 \\
0 & \al t_1^{-1} & \al & \al t_2^{-1} & 0 & 0 & 0 \\
0 & 0 & \al & \al t_2^{-1} & 0 & 0 & 0 \\
0 & 0 & \al & \al t_2^{-1} & \al t_2^{-1}t_1^{-1} & 0 & \al t_1^{-1} \\
0 & 0 & \al & 0 & \al t_2^{-1} & \al &\al t_1^{-1} \\
0 & 0 &	0 & 0 &	0 & \al & \al t_1^{-1} \\
0 & 0 & 0 & 0 & 0 & c_2 & c_3
\end{pmatrix}\,,
\]
where~$c_1,c_2$ and~$c_3$ are column vectors.
For~$1 \leq i \leq 3$, let~$A_i$ be~$F_0^{q}$ with the third (local) line and the column containing~$c_i$ removed and set~$B_{i}(t)=\mbox{det}(A_{i}(t^\mathbf{2}))$.  Finally, set
$\beta_0=\frac{\sgn(x_0)(-1)^{u_0}}{\omega_0 t^{\kappa_0}}\frac{1}{t_1^2-1}$. The values of the functions~$\Gamma$ will be computed by using row operations such that the last column of each local matrix is zero, except for its two lowest entries. 

Clearly,~$\kappa_0$,~$\kappa_{++}$ and~$\kappa_{--}$ coincide,~$u_0$,~$u_{++}$ and~$u_{--}$ all have the same parity, while~$\sgn(x_0)=-\sgn(x_{++})=\sgn(x_{--})$. Moreover, 
the corresponding weights satisfy the equalities~$\omega_{++}=\omega_0t_1t_2^{-1}$ and~$\omega_{--}=\omega_0\al^{8}t_1^{-3}t_2^{-5}$.  Finally, for~$F_0^{q}$, replace the third (local) line by its difference with the second (local) line, expanding with respect to the third (local) line yields~$\Gamma_0 = (-1)^{\ell}\beta_0\al^2B_2$. The others are obtained similarly as:
\[
\Gamma_{++}= (-1)^{\ell}\beta_0\al^2t_1^{-1}t_2(-(1-t_1^{-2})B_1 + t_2^{-2}B_2),
\hspace{0.3cm}
\Gamma_{--}=(-1)^{\ell}\beta_0\al^2t_1t_2(-( t_1^{-4}-t_1^{-2})B_1 + B_2)\,.
\]
It follows easily that~$\Gamma_{++} + \Gamma_{--}=(t_1t_2 + t_1^{-1}t_2^{-1})\Gamma_0$.

\smallskip

\textbf{(R5)} Let us suppose that~$i=1$,~$j=2$, and~$k=3$. The diagrams considered are the following,
\medskip
\begin{center}
\includegraphics{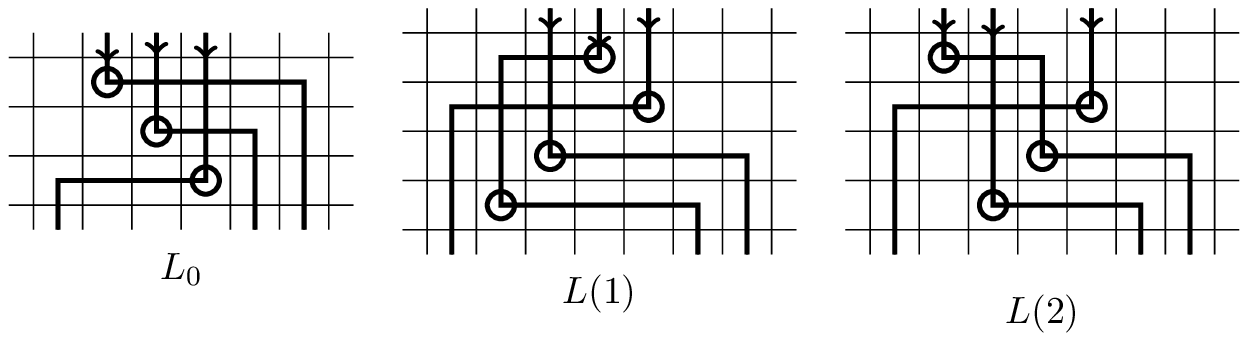}
\end{center}
\begin{center}
\includegraphics{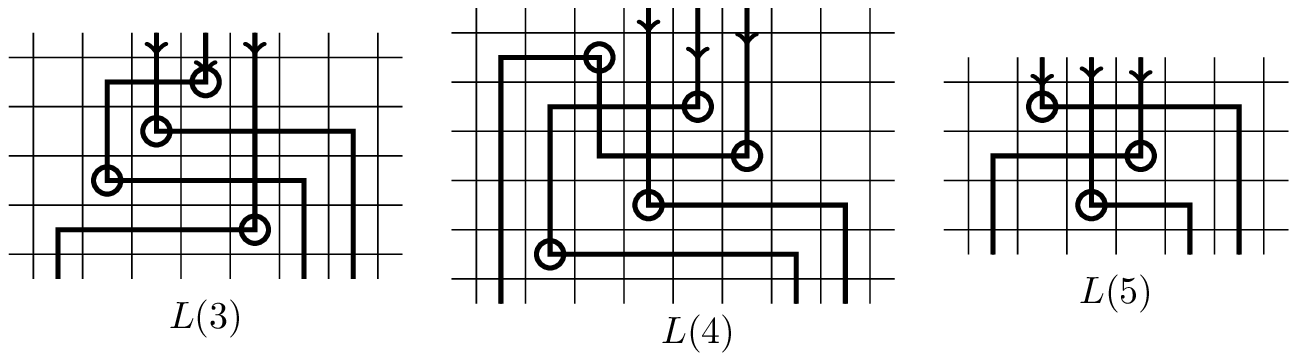}
\end{center}

\noindent defining local Fox matrices of the form:
\[
F_0^{q+5}=
\begin{pmatrix}
0   & c_2 & c_3 &    c_4       & 0          \\
0   & \al & \al t_1^{-1} &    \al t_1^{-1}t_2^{-1}       & 0           \\
0   &  0  & \al & \al t_2^{-1} & 0           \\
0   &  0  &  0  &    \al       & \al t_3^{-1} \\
\al &  0  &  0  &    0         & \al t_3^{-1}         \\
c_1 &  0  &  0  &    0         & c_5
\end{pmatrix},
F_1^{q+6}=
\begin{pmatrix}
0   & 0             & c_2                  & c_3 & c_4                 & 0   \\
0   & 0             & \al                  & \al t_1^{-1} & \al t_1^{-1}t_2^{-1}                 & 0   \\
0   & \al           & \al t_2^{-1}         & 0   & \al t_2^{-1}t_1^{-1}& 0   \\
\al & \al t_3^{-1}  & \al t_3^{-1}t_2^{-1} & 0   & 0                   & 0   \\
\al & \al t_3^{-1}  & 0                    & 0   & 0                   & 0   \\
\al & 0             & 0                    & 0   & 0                   & \al t_3^{-1} \\
c_1 & 0             & 0                    & 0   & 0                   & c_5 \\
\end{pmatrix},
\]
\[
F_2^{q+6}=
\begin{pmatrix}
0   & c_1 & c_2         & 0                   & c_3                 & 0   \\
0   & 0   & \al         & \al t_2^{-1}        & \al t_2^{-1}t_1^{-1}& 0   \\
\al & 0   & \al t_3^{-1}& \al t_3^{-1}t_2^{-1}& 0                   & 0   \\
\al & 0   & \al t_3^{-1}& 0                   & 0                   & 0   \\
\al & 0             & 0                    & 0   & 0                   & \al t_3^{-1} \\
c_1 & 0             & 0                    & 0   & 0                   & c_5 \\
\end{pmatrix},
F_3^{q+6}=
\begin{pmatrix}
0   & 0             & c_2                  & c_3 & c_4                 & 0   \\
0   & 0             & \al                  & \al t_1^{-1} & \al t_1^{-1}t_2^{-1}                 & 0   \\
0   & \al & \al t_2^{-1}& 0      & \al t_2^{-1}t_1^{-1}& 0             \\
0   & \al & 0           & 0      & \al t_2^{-1}        & 0             \\
0   & 0   & 0           & 0      & \al                 & \al t_3^{-1}  \\
\al & 0             & 0                    & 0   & 0                   & \al t_3^{-1} \\
c_1 & 0             & 0                    & 0   & 0                   & c_5 \\
\end{pmatrix},
\]
\[
F_4^{q+7}=
\begin{pmatrix}
0   & 0             &                    0 & c_2                  & c_3 & c_4                 & 0   \\
0   & 0             & 0                    & \al                  & \al t_1^{-1} & \al t_1^{-1}t_2^{-1}                 & 0   \\
\al & 0             & \al t_3^{-1}        & \al         & \al t_1^{-1} & \al t_1^{-1}t_2^{-1} & 0 \\
\al & \al t_3^{-1}  & \al t_3^{-1}t_2^{-1}& \al t_2^{-1}& 0      & \al t_2^{-1}t_1^{-1}        & 0 \\
\al & \al t_3^{-1}  & 0                   & \al t_3^{-1}t_2^{-1} & 0      & 0                 & 0 \\
\al & \al t_3^{-1}  & 0                   & 0           & 0      & 0                   & 0       \\    
\al & 0             & 0                    & 0   & 0  &0                 & \al t_3^{-1} \\
c_1 & 0             & 0                    & 0   & 0  &0                & c_5 \\
\end{pmatrix},
F_5^{q+5}=
\begin{pmatrix}
0   & c_2 & c_3 &    c_4       & 0          \\
0   & \al & \al t_1^{-1} &    \al t_1^{-1}t_2^{-1}       & 0           \\
0   &  0  & \al          & \al t_2^{-1} & 0           \\
\al &  0  & \al t_3^{-1} &    0         & 0           \\
\al &  0  &  0  &    0         & \al t_3^{-1}         \\
c_1 &  0  &  0  &    0         & c_5
\end{pmatrix}\,.
\]
For~$1\le i,j\le 5$, let~$A_{ij}$ be~$F_0^{q+5}$ with the third and fourth (local) lines, as well as the columns containing
$c_i$ and~$c_j$ removed, and set~$B_{ij}(t)=\mbox{det}(A_{ij}(t^\mathbf{2}))$. Finally, set
$\beta_0=\frac{\sgn(x_0)(-1)^{u_0}}{\omega_0 t^{\kappa_0}}\frac{(-1)^{q+5}}{t_1^2-1}$. The values of the~$\Gamma$ functions are computed by using row operations such that the first column of each local matrix is zero, except for its two lowest entries. 

First note that neither the curvature nor the number of upwards segments change, except for~$L(4)$ where~$u_4=u_0+1$.
Moreover, the sign of the base changes as follows:~$\sgn(x_0)=(-1)^{\ell+q}\sgn(x_1)=(-1)^{\ell+q}\sgn(x_2)=(-1)^{\ell+q}\sgn(x_3)=-\sgn(x_4)=-\sgn(x_5)$. Note that there is an additional sign for~$L(1), L(2),L(3)$ due to the fact that the removed columns do not have the same parity as~$q+5$.
The weights satisfy~$\omega_1=\omega_0\al^2t_1^{-2}t_2^{-2}t_3^{-2}$,
$\omega_2=\omega_0\al^2t_1^{-1}t_2^{-3}t_3^{-2}$,~$\omega_3=\omega_0\al^2t_1^{-1}t_2^{-1}$,~$\omega_4=\omega_0\al^4t_1^{-2}t_2^{-3}t_3^{-5}$, and
$\omega_5=\omega_0t_2^{-1}t_3^{-1}$. For~$F_0^{q+5}$, replace the fourth (local) line by its difference with the third (local) line, then the third (local) line by its difference with the second (local) line, leading to:
\[
\begin{array}{lll}
\Gamma_{L_0} &= &\beta_0\al^4\cdot(B_{23} - (t_2^{-2}-1)B_{24} -t_3^{-2}B_{25} -(t_1^{-2}-1)B_{34} + t_3^{-2}(t_1^{-2}-1)B_{35} \\
& & - t_2^{-2}t_3^{-2}(t_1^{-1}-1)B_{45})\,.\\
\end{array}
\]
The others are computed similarly:
\[
\begin{array}{lll}
\Gamma_{L(1)} &= &\beta_0\al^4t_1^{2}t_3^{-2}\cdot( t_1^{-2}B_{23}- B_{25})\\

\Gamma_{L(2)}&=&\beta_0\al^4t_1t_2t_3^{-2}\cdot(B_{12} -B_{25}+(t_1^{-2}-1)B_{35})\\

\Gamma_{L(3)} &=& \beta_0\al^4t_1t_2^{-1}\cdot(t_1^{-2}B_{23}-t_1^{-2}(t_2^{-2}-1)B_{24}-t_3^{-2}B_{25}-t_1^{-2}(t_1^{-2}-1)B_{34} \\
& &-t_3^{-2}(t_1^{-2}-1)B_{45})\\

\Gamma_{L(4)} &=& \beta_0\al^4t_1^{2}t_2^{-1}t_3^{-1}\cdot(t_1^{-2}B_{23} - t_1^{-2}(t_2^{-2}-1)B_{24}-t_3^{-2}B_{25})\\

\Gamma_{L(5)} &=& \beta_0\al^4t_2t_3^{-1}\cdot(B_{23} -B_{25} -t_2^{-2}(t_1^{-2}-1)B_{34} + (t_1^{-2}-1)B_{35} -t_2^{-2}(t_1^{-2}-1)B_{45})\,.
\end{array}
\]
Letting ~$c_{ij}^k$ be the coefficient of~$B_{ij}$ in~$\Gamma_{L(k)}$ (divided by~$\beta_0\al^4$), it suffices to show that
\[
\Gamma_{ij} := (t_1t_3^{-1}-t_1^{-1}t_3)(c_{ij}^{0} + c_{ij}^{1}) 
			  + (t_2t_3 - t_2^{-1}t_3^{-1})(c_{ij}^{2} + c_{ij}^{3}) 
			+(t_1^{-1}t_2^{-1}-t_1t_2)(c_{ij}^{4}+c_{ij}^5)
\]
vanishes for all~$i,j \in \{2,3,4,5\}$. For example~$B_{24}$ appears only in~$L(0)$,~$L(3)$ and~$L(4)$, and we compute
\begin{eqnarray*}
\Gamma_{24} &=& (t_1t_3^{-1}-t_1^{-1}t_3)( -(t_2^{-2}-1) + (t_2t_3 - t_2{-1}t_3^{-1})(-t_1^{-1}t_2^{-1}(t_2{-2}-1))\\
			& & + (t_1^{-1}t_2^{-1}-t_1t_2)(-t_2^{-1}t_3^{-1}(t_2{-2}-1))=0\,.
\end{eqnarray*}
The other equations can be verified in the same way through direct computation.
\end{proof}


\section{Link Floer homology and Euler characteristic}\label{grid}

This section gives a very quick review of the combinatorial version of link Floer homology with~$\bbZ_2$ coefficients
following~\cite{manolescu2006combinatorial,manolescu2007combinatorial}, and mentions selected properties. We then use these properties to show that
link Floer homology categorifies the function~$\Gamma_L$ introduced in the previous section. Theorem~\ref{mresult} follows.

\subsection{Grid diagrams and gradings}

Given an oriented rectangular diagram, consider each corner labeled by an~$O$ or an~$X$ with the following convention: for each horizontal segment we travel from an “$O$” to an “$X$” whereas for each vertical segment we travel from an “$X$” to an “$O$”. One obtains a \emph{grid diagram} by fitting the data of those~$X$'s and~$O$'s in a grid. Due to condition (iii), each row and column of the grid contain exactly one~$X$ and one~$O$. 
From a grid diagram, one can construct a rectangular diagram by drawing horizontal segments oriented from~$O$ to~$X$ and vertical segments oriented from~$X$ to~$O$.
Rectangular and grid diagrams are thus equivalent, and the elementary moves for grid diagrams are the same as for rectangular ones. Examples are given in Figure~\ref{exx0}.

Note that the~$O$'s lie in the exact same place as the circles defining the base of a rectangular diagram.
We denote by~$\bbX$ the set of all~$X$'s and by~$\bbO$ the set of all~$O$'s of the diagram. 

\begin{figure}[Htb]
\begin{center}
\includegraphics{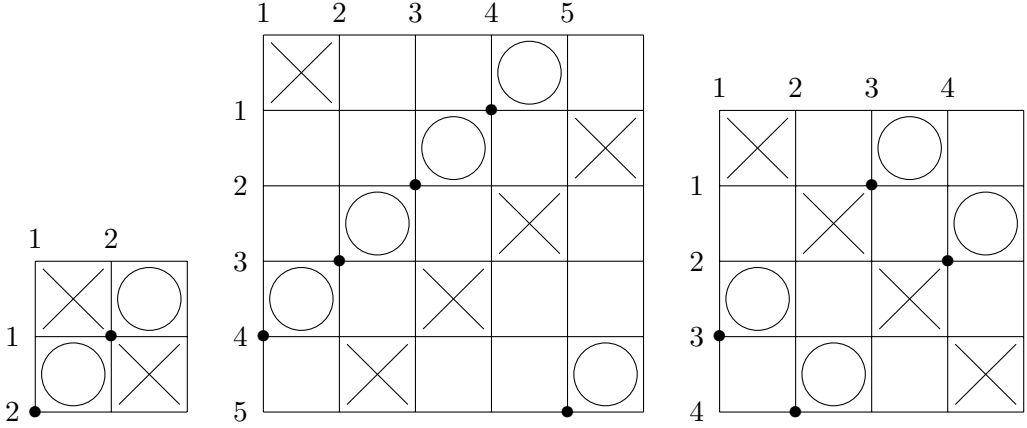}
\caption{The three grid diagrams corresponding to the three rectangular diagrams of Figure~\ref{gammacalc}.}
\label{exx0}
\end{center}
\end{figure}

\medskip

Given a grid diagram~$G$ of size~$n$, number the horizontal lines from~$1$ to~$n$, starting with the second line, and the vertical lines from~$1$ to~$n$, starting with the leftmost line (see Figure \ref{exx0}). Let~$S(G)$ (or simply~$S$) be the set of all~$n$-tuples of intersection points between horizontal and vertical lines, with the property that no intersection point appears on more than one horizontal or vertical line. There is an element of~$S(G)$ which corresponds exactly to the base~$x_0$ of the diagram: it is given by
the~$n$-tuple which occupies the lower left corner of each square containing an~$O$ (indicated in Figure~\ref{exx0} as a collection of black dots).

\medskip

The set~$S$ is equipped with two gradings.
The \emph{Maslov grading}~$M\colon S\longrightarrow \bbZ$ is defined as follows. Given two collections~$A, B$ of finitely many points in the plane,
let~$I(A,B)$ be the number of pairs~$(a_1,a_2) \in A$ and~$(b_1,b_2) \in B$ such that~$a_1<b_1$ and~$a_2<b_2$ and set~$J(A,B)=(I(A,B)+I(B,A))/2$.
Given an element~$x \in S$, we view~$x$ as a collection of points with integer coordinates. Similarly the sets~$\bbO=\{O_i\}_{i=1}^{n}$ and~$\bbX=\{X_i\}_{i=1}^{n}$ are viewed as a collection of points with half-integer coordinates. Define 
\begin{equation*}
M(x)=J(x-\bbO,x-\bbO)+1\,,
\end{equation*}
where~$J$ is extended bilinearly over formal sums and differences of subsets.
Note that the Maslov index of the base equals~$1-n$.

For a link of multiplicity~$\mu$, the \emph{Alexander grading} is the function~$A: S \longrightarrow \left(\frac{1}{2}\bbZ\right)^{\mu}$ defined as a~$\mu$-tuple~$A(x)=(A_1(x),\ldots,A_{\mu}(x))$ by the formula
\[
A_k(x)=J\left(x-\frac{1}{2}(\bbX+\bbO),\bbX_k - \bbO_k \right)- \left(\frac{n_k-1}{2}\right)\,,
\]
where~$\bbO_k \subset \bbO$ and~$\bbX_k \subset \bbX$ are the subsets corresponding to the~$k^\mathrm{th}$ component of the link,~$n_k$ is the cardinal of~$\bbO_k$, and~$J$ is again extended bilinearly.

\begin{lemma}\label{sign}
For any two element~$x,y \in S$, we have~$(-1)^{M(x)}\,\sgn(x)= (-1)^{M(y)}\, \sgn(y)$.
\end{lemma}

\begin{proof}
By additivity of the signature, we only need to check that if the permutations~$x$ and~$y$ differ by a transposition, then~$M(x)-M(y)$ is odd. Using the terminology of~\cite{manolescu2007combinatorial},
this means that there is an empty rectangle~$r$ connecting~$x$ and~$y$ (or~$y$ and~$x$). By Lemma 2.5 of that same article, we have
\[
\pm (M(x)-M(y))=1-2\#(\bbO\cap r)\,,
\]
and the lemma follows.
\end{proof}

\subsection{Combinatorial link Floer homology}

In this paragraph, we recall the main properties of the combinatorial version of link Floer homology over~$\bbZ_2$ following~\cite{manolescu2007combinatorial},
displaying the ones that will be needed in the sequel.

\medskip

Let~$G$ be a grid diagram. Let~$R$ denote the polynomial algebra over~$\bbZ_2$ generated by variables~$\{U_i\}_{i=1}^{n}$ which are in one-to-one correspondence with the elements
of~$\bbO$. This ring is endowed with a \emph{Maslov grading}, defined such that the constant terms are in Maslov grading zero and the~$U_i$'s in grading~$-2$. It is also endowed with an \emph{Alexander multi-filtration}, defined so that the constants are in filtration level zero while the variables~$U_j$ corresponding to the~$i^\mathrm{th}$ component drop the~$i^\mathrm{th}$ multi-filtration level by one and preserve the others.

Let~$C^{-}(G)$ be the free~$R$-module with generating set~$S$. It has a Maslov grading and an Alexander filtration induced by the ones on~$S$ and~$R$.
It turns out that one can define an endomorphism~$\partial^{-}\colon C^{-}(G)\longrightarrow C^{-}(G)$ which is a differential, drops the Maslov grading by one and
preserves the Alexander filtration level. Furthermore, this complex provides a link invariant as follows.
Let~$L$ be an oriented link given by a grid diagram~$G$. Number the elements of~$\bbO=\{O_i\}_{i=1}^{n}$ so that~$O_1,\ldots,O_{\mu}$ correspond to the different components
of the link. Then the filtered quasi-isomorphism type of the complex~$(C^{-}(G),\partial^{-})$ seen as a module over~$\bbZ_2[U_1,\ldots,U_{\mu}]$ is a link
invariant~\cite[Theorem 1.2]{manolescu2007combinatorial}. In particular, the homology of the associated graded object is a link invariant. (We refer to subsection~2.1 of~\cite{manolescu2007combinatorial} for an explanation of this algebraic terminology.)

There are two variations on this construction. With an ordering of the~$O$'s as before, consider the chain complex~$\widehat{C}(G)$ obtained from~$C^{-}(G)$ by setting
$U_i=0$ for~$i=1,\ldots, \mu$. Let~$\widehat{\mathit{CL}}(G)$ denote the graded object associated to the Alexander filtration and let~$\widehat{\mathit{HL}}(G)$ be its homology. 
Also, let~$\widetilde{C}(G)$ be the chain complex obtained from~$C^{-}(G)$ by setting all the~$U_i$'s to~$0$. Let~$\widetilde{\mathit{CL}}(G)$ be the associated graded object and
let~$\widetilde{\mathit{HL}}(G)$ be its homology. These two complexes can be related using standard tools of homological algebra~\cite[Proposition 2.15]{manolescu2007combinatorial}.

\begin{proposition}
\label{hattild}
The homology groups~$\widehat{\mathit{HL}}(G)$ determine~$\widetilde{\mathit{HL}}(G)$ as follows:
\[
\widetilde{\mathit{HL}}(G) \cong \widehat{\mathit{HL}}(G) \otimes \bigotimes\limits_{k=1}^{\mu} V_k^{\otimes(n_k-1)}\,,
\]
where~$V_k$ is the two-dimensional vector space spanned by two generators, one in zero Maslov and Alexander multi-gradings, and the other in Maslov grading minus one and Alexander multi-grading corresponding to  minus the~$k^\mathrm{th}$ basis vector, and~$n_k$ is the number of~$O$'s corresponding to~$L_k$.\qed
\end{proposition}

We shall need one last property of link Floer homology~\cite[Proposition 5.3]{manolescu2007combinatorial}.

\begin{proposition}
\label{prop:LFH2}
Let~$\widehat{\mathit{HL}}_d(L,s)$ denote the subspace of~$\widehat{\mathit{HL}}(L)$ with Maslov grading~$d$ and Alexander grading~$s$.
Then,~$\widehat{\mathit{HL}}_d(L,s)$ and~$\widehat{\mathit{HL}}_{d-\sum_k s_k}(L,-s)$ are isomorphic for all~$d$,~$s$.\qed
\end{proposition}

\subsection{Euler Characteristic and~$\Gamma_L$}

Given a multi-graded vector space~$C=\bigoplus_{d,s}C_d(s)$ with Maslov grading~$d$ and Alexander grading~$s$, define its \emph{(graded) Euler characteristic} as
\[
\chi(\widehat{\mathit{HL}}(L);t):= \sum\limits_{d,s} (-1)^d t^s \mbox{dim}\,C_d(s)\,.
\]
Propositions~\ref{hattild} and~\ref{prop:LFH2} immediately imply:

\begin{corollary}
\label{cor:chi}
For any oriented link~$L$, we have the equalities in~$\bbZ[t_1^{\pm 1/2},\ldots,t_{\mu}^{\pm 1/2}]$:
\begin{eqnarray*} 
\chi(\widetilde{\mathit{HL}}(L);t)&=&\prod\limits_{k=1}^{\mu}(1-t_k^{-1})^{n_k-1}\,\chi(\widehat{\mathit{HL}}(L);t)\,,\\
\chi(\widehat{\mathit{HL}}(L);t)&=&\chi(\widehat{\mathit{HL}}(L);t^{\bf{-1}})\,.\qed
\end{eqnarray*}
\end{corollary}

Theorem~\ref{mresult} now follows from Proposition~\ref{gammavsnabla} and the following result.

\begin{proposition}
\label{chivsgamma}
For any oriented link~$L$, we have the equality
\[
\chi(\widehat{\mathit{HL}}(L);t)=\prod\limits_{k=1}^{\mu}(t_k^{1/2}-t_k^{-1/2}) \, \Gamma_L(t^{\mathbf{1/2}})\,.
\]
\end{proposition}

Before giving the proof, let us define one more matrix following~\cite[Section 6]{manolescu2007combinatorial}. 
Given a rectangular diagram~$D$ of size~$n$, (or equivalently, a grid diagram~$G$), define an~$n \times n$ matrix~$W(G)$ by 
\[
W(G)_{i,j}=t^{a(i,j)}\,,
\]
where~$a(i,j)\in\bbZ^{\mu}$ is the vector whose~$k^\mathrm{th}$ component is minus the winding number of~$L_k$ around the point~$(i,j)$,
counted from the second horizontal line and first vertical line. Note that this definition implies that~$W(G)_{n,j}=W(G)_{i,1}=1$ for all~$i,j$, as these points lie outside
the link. An easy example for the trivial knot is given below: the function~$a(i,j)$ is shown on the left and the matrix~$W(G)$ on the right.
\begin{center}
\includegraphics{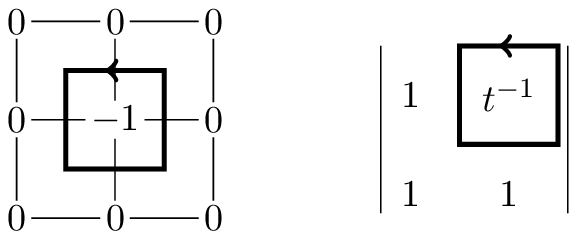}
\end{center}

\begin{proof}[Proof of Proposition~\ref{chivsgamma}]
Let~$D$ be a rectangular diagram representing~$L$. Let us first compare the determinant of~$W(G)$ with~$\Gamma_L=\Gamma_D$.
In~$W(G)$, subtract each column from the next one. In every column but the first of the resulting matrix, there is a zero where the vertical segment does not intervene, and where a vertical segment is present, the coefficient is divisible by~$\theta(x_j) - 1$. Therefore, for each column but the first, we can factor out~$\theta(x_j)-1$.
The last vertical segment which corresponds to the generator~$x_n$ does not contribute to these factors, as it does not lie in between two columns of the matrix.
Moreover, the last line is identically zero except for the first coefficient, which is one. Therefore,
\[
\mbox{det}(W(G))= (-1)^{n+1}\prod\limits_{j=1}^{n}(\theta(x_j)-1) \, \frac{\mbox{det}(A)}{\theta(x_n)-1}\,,
\]
where~$A$ is the~$(n,1)$-minor of the resulting matrix, a minor which is nothing but~$F_D^n$ (recall paragraph~\ref{rect}).
Define~$u_k$ as the number of upwards segments of~$L_k$, and let~$d$ be the total number of downward segments.
Transforming each term~$\theta(x_j)-1$ into~$1-t_k^{-1}$ by factoring out a sign if~$x_j$ is downwards and~$t_k$ if it is upwards, we get
\[
\mbox{det}(W(G))=(-1)^{1+d}\prod\limits_{k=1}^{\mu}t_k^{u_k}\prod\limits_{k=1}^{\mu}(1-t_k^{-1})^{n_k} \, \frac{(-1)^n \mbox{det}(F_D^n)}{\theta(x_n)-1}\,.
\]
Set~$\gamma\in(\frac{1}{2}\bbZ)^\mu$ such that~$t^{2\gamma}=\omega t^{\kappa}$ and~$\beta=\gamma + (u_1,\ldots,u_{\mu})$.
Multiplying both sides by the sign~$(-1)^{M(x_0)}\sgn(x_0)$ and using the equalities~$M(x_0)=1-n=1-d-u$, we eventually get
\begin{equation}
\label{WGamma}
(-1)^{M(x_0)}\sgn(x_0)\,\mbox{det}(W(G))=t^{\beta}\prod\limits_{k=1}^{\mu}(1-t_k^{-1})^{n_k} \, \Gamma_D(t^{\mathbf{1/2}})\,.
\end{equation}

Let us now relate the determinant of~$W(G)$ with~$\chi(\widehat{\mathit{HL}}(L);t)$. By Corollary~\ref{cor:chi},
\[
\prod\limits_{k=1}^{\mu}(1-t_k^{-1})^{n_k-1} \chi(\widehat{\mathit{HL}}(L);t)= \chi(\widetilde{\mathit{HL}}(L);t)=\chi(\widetilde{C}(G);t)\,.
\]
Furthermore, using the definitions of~$\widetilde{C}(G)$ and of~$A(x)$ together with Lemma~\ref{sign}, we have
\[
\begin{array}{ccl}

\chi(\widetilde{C}(G);t) &=& \sum_{d,s} (-1)^d t^s \mbox{dim}(\widetilde{C}_d(s)) \\
   & = & \sum_{d,s} (-1)^d t^s \# \{ x \in S \, | \, M(x)=d \mbox{ and } A(x) = s \} \\
   & = & \sum_{x \in S} (-1)^{M(x)} t^{A(x)} \\
   & = & (-1)^{M(x_0)} \sgn(x_0)\,\sum_{x \in S} \sgn(x)\,t^{A(x)} \\
   & = & (-1)^{M(x_0)} \sgn(x_0)\,t^{\nu +\frac{1}{2}} \sum_{x \in S} \sgn(x)\prod_{k=1}^{\mu}t_k^{J(x,\bbX_k - \bbO_k)}\,,
\end{array}
\]
where~$\nu_k=J(-\frac{1}{2}(\bbX+\bbO),\bbX_k - \bbO_k) -\frac{n_k}{2}$,~$\nu=(\nu_1,\ldots, \nu_{\mu})$ and~$\bbX_k$ and~$\bbO_k$ are the markings corresponding~$L_k$.
It is straightforward to see that~$J(x,\bbX_k - \bbO_k)$ is minus the sum of the winding numbers of~$L_k$ around the points in~$x$. Therefore 
\[
\displaystyle \sum\limits_{x \in S} \sgn(x)\prod\limits_{k=1}^{\mu}t_k^{J(x,\bbX_k - \bbO_k)} = \mbox{det}(W(G))\,,
\]
and we eventually obtain
\begin{equation}
\label{Wchi}
(-1)^{M(x_0)}\sgn(x_0)\,\mbox{det}(W(G))=t^{-\nu-\frac{1}{2}}\prod\limits_{k=1}^{\mu}(1-t_k^{-1})^{n_k-1} \, \chi(\widehat{\mathit{HL}}(L);t)\,.
\end{equation}

Equations~(\ref{WGamma}) and~(\ref{Wchi}) lead to
\[
\chi(\widehat{\mathit{HL}}(L);t) = t^{\beta+\nu} \displaystyle \prod\limits_{k=1}^{\mu}(t_k^{1/2}-t_k^{-1/2}) \, \Gamma_D(t^{\mathbf{1/2}})\,.
\]
The fact that~$\beta + \nu$ vanishes now follows from Equation~(\ref{gammasym}) and the second part of Corollary~\ref{cor:chi}.
This concludes the proof.
\end{proof}


\section{Other models and homology spheres}
\label{models}

As mentioned in the introduction, there are four different constructions of a well-defined representative of the multivariable Alexander polynomial. In this slightly informal section,
we discuss the identification of these models and of their extension to integral homology spheres.

\medskip

Let us start by recalling these four models.

\begin{enumerate}
\item{As explained in paragraph~\ref{sub:hartley}, Hartley's model~$\nabla_L$ is based on Fox free differential calculus and the Wirtinger presentation of the link group~\cite{hartley1983conway}.}
\item{Turaev's model~$\tau_L$ is constructed using sign-refined Reidemeister torsion~\cite{turaev1986reidemeister}.}
\item{The second author's model~$\Omega_L$ is defined via generalized Seifert surfaces~\cite{cimasoni2004geometric}.}
\item{Finally, Ozsváth and Szabó's link Floer homology~\cite{ozsvath2004holomorphic} provides one last model in the form of its renormalized Euler characteristic
$\chi_L(t):=\prod_k(t_k-t_k^{-1})^{-1}\chi(\widehat{\mathit{HL}}(L);t^\mathbf{2})$.}
\end{enumerate}

All these models being symmetric representatives of the Alexander polynomial, they agree up to sign. Theorem~\ref{mresult} states that~$\nabla_L$ and~$\chi_L$ coincide,
while the equality~$\Omega_L=\nabla_L$ is checked in~\cite{cimasoni2004geometric}. There is no published proof of the fact that~$\tau_L$ agrees with~$\nabla_L$, so let us
explain one way of verifying this fact. In~\cite[Section 4]{turaev1986reidemeister}, Turaev gave a list of six axioms satisfied by~$\tau_L$ and characterizing this invariant
for links in~$S^3$. It is not difficult to check that the geometric model~$\Omega_L$ satisfies these axioms: the first five follow from~\cite{cimasoni2004geometric}
(together with the proof of~\cite[Proposition 2.5]{C-F}), while the last one -- the so-called {\em doubling axiom\/} -- can be verified using the same elementary homological techniques. Hence, all
four models coincide for links in~$S^3$.

\medskip

Now note that while Hartley's model~$\nabla_L$ is only defined for links in the standard sphere~$S^3$, the definition of the other three models can be extended \emph{verbatim}
to links in arbitrary integral homology spheres. Do they coincide in this more general setting? To address this question, let us recall two properties of~$\nabla_L$.

\bigskip

\noindent{\bf Torres formula}~\cite{Torres,hartley1983conway}. If~$L=L_1\cup\dots\cup L_\mu$ and~$L^*=L\cup L_{\mu+1}\cup\dots\cup L_n$, then
\[
\nabla_{L^*}(t_1,\dots,t_\mu,1,\dots,1)=f_{L^*,L}(t_1,\dots,t_\mu)\,\nabla_L(t_1,\dots,t_\mu)\,,
\]
where~$f_{L^*,L}(t_1,\dots,t_\mu)=\prod_{i=\mu+1}^n(t_1^{\ell_{1i}}\cdots t_\mu^{\ell_{\mu i}}-t_1^{-\ell_{1i}}\cdots t_\mu^{-\ell_{\mu i}})$ and~$\ell_{ij}=\mathrm{lk}(L_i,L_j)$.

\bigskip

\noindent{\bf Variance under surgery}~\cite{boyer1992conway}. Let~$L$ be a framed link in~$S^3$ such that the manifold obtained by surgery along~$L$ is again~$S^3$. Then,
\[
\nabla_{\widehat{L}}(t)=\det(B)^{-1}\,\nabla_L(t\cdot B^{-1})\,,
\]
where~$\widehat{L}$ is the link given by the cores of the surgery tori,~$B$ is the framing matrix associated to~$L$, and~$t\cdot A$ stands for 
$(t_1^{a_{11}}t_2^{a_{21}}\cdots t_\mu^{a_{\mu 1}},\dots,t_1^{a_{1\mu}}t_2^{a_{2\mu}}\cdots t_\mu^{a_{\mu\mu}})$ if~$A=(a_{ij})$.

\bigskip

The following proposition is an easy consequence of the work of Boyer-Lines~\cite{boyer1992conway}.

\begin{proposition}
Let~$\nabla$ and~$\nabla'$ be two invariants of links in homology spheres such that:
\begin{enumerate}
\item{$\nabla_L=\nabla'_L$ for any link~$L$ in~$S^3$,}
\item{$\nabla$ and~$\nabla'$ satisfy the Torres formula for any link in a homology sphere, and}
\item{$\nabla$ and~$\nabla'$ satisfy the variance under surgery formula for any framed link in~$S^3$ such that the surgery manifold is a homology sphere.}
\end{enumerate}
Then,~$\nabla$ and~$\nabla'$ coincide for any link in a homology sphere.
\end{proposition}
\begin{proof}
Let~$L$ be a link in a homology sphere~$\Sigma$. By~\cite[Lemma 2.3]{boyer1992conway}, there exists a link~$L^*\subset\Sigma$ containing~$L$ as a sublink, such that
$f_{L^*,L}(t)$ does not vanish and~$\Sigma\setminus L^*$ is homeomorphic to~$S^3\setminus L^0$ for some link~$L^0$ in~$S^3$. The image of the meridians under this homeomorphism
defines a framing of~$L^0$ such that the resulting surgery manifold is~$\Sigma$ and the resulting link~$\widehat{L^0}$ is~$L^*$. By the first and third assumption, we have
\[
\nabla_{L^*}(t^*)=\det(B)^{-1}\,\nabla_{L^0}(t^*\cdot B^{-1})=\det(B)^{-1}\,\nabla'_{L^0}(t^*\cdot B^{-1})=\nabla'_{L^*}(t^*)\,.
\]
Using this equality together with the second assumption, we get
\[
f_{L^*,L}(t)\nabla(t)=\nabla_{L^*}(t,1)=\nabla'_{L^*}(t,1)=f_{L^*,L}(t)\nabla'_L(t)\,,
\]
and the conclusion follows since~$f_{L^*,L}(t)$ does not vanish.
\end{proof}

This result can be used to identify~$\Omega_L$ and~$\tau_L$. Indeed, we already know that they coincide for links in~$S^3$. The Torres formula for~$\tau_L$ is a special case of
Theorem 1.4 in~\cite[Chapter VII]{Tu}, while it can be checked for~$\Omega_L$ using standard homological computations. As for the variance of~$\tau_L$ under surgery,
we have found no proof
in the literature but several related results are known (see e.g.~\cite[Chapter VIII]{Tu} and~\cite{Cim}) and the same methods apply to give the desired result.
Unlike~$\tau_L$, the geometrical model~$\Omega_L$ is unfortunately not suited for surgery considerations. However, the variance under surgery can be deduced from several
simpler properties~\cite{Boyer} which can be checked for~$\Omega_L$. In conclusion,~$\Omega_L$ and~$\tau_L$ coincide for any link in any homology sphere.

We have not been able to find in the literature a result in link Floer homology that categorifies the variance of~$\chi_L$ under surgery. For this reason, we do not know if this method applies to identify~$\chi_L$ and~$\tau_L$ is the general setting of homology spheres. Alternatively, one could try to identify these two models working directly from the
definition.

\bibliographystyle{amsalpha}
\bibliography{bibli}
\end{document}